\documentclass[11pt]{amsart}

\usepackage[margin=2cm]{geometry}
\usepackage{amsmath, amsthm, xcolor}
\usepackage{amssymb}
\usepackage{graphicx,soul}
\usepackage[normalem]{ulem}
\usepackage{relsize}
\usepackage{enumerate}

\newcommand{\hidden}[1]{}

\usepackage{color}
\usepackage[mathscr]{euscript}
\usepackage{tikz}

\numberwithin{equation}{section}

\newcommand{\R}{{\Bbb R}}

\newcommand{\Z}{{\Bbb Z}}
\newcommand{\N}{{\Bbb N}}

\newtheorem{theorem}{Theorem}[section]

\newtheorem{corollary}{Corollary}

\theoremstyle{remark}
\newtheorem{remark}{{Remark}}

\def\N{\mathbb N}

\def\R{\mathcal R}

\def\Z{\mathbb Z}

\def\Z{\mathbb Z}

\newtheorem{lem}{\bf Lemma}[section]

\def\N{\mathbb N}
\def\N{\mathbb N}

\def\R{\mathbb R}
\def\Z{\mathbb Z}

\newcommand{\dimH}{\dim_{H}}
\begin{document}

\title{The Cartesian product of exact approximation sets}

\author{Jing Guo}
\address{Jing Guo, School of Mathematics and Statistics, Huazhong University of Science and Technology, Wuhan 430074, P. R. China}
\email{guojing\_math@hust.edu.cn}

\author[Mumtaz Hussain]{Mumtaz Hussain}
\address{Mumtaz Hussain, Department of Mathematical and Physical Sciences,  La Trobe University, Bendigo 3552, Australia. }
\email{m.hussain@latrobe.edu.au}

\author{Bixuan Li}
\address{Bixuan Li, Institute of Mathematics, Academy of Mathematics and System Science, CAS, Zhongguancun East Road 55, Beijing 100190, P. R. China}
\email{math\_forever@163.com}

\author[Johannes Schleischitz]{Johannes Schleischitz}
\address{Johannes Schleischitz, School of Computer, Data, and Mathematical Sciences, Western Sydney University, NSW, Australia}
\email{J.Schleischitz@westernsydney.edu.au}

\date{}

\begin{abstract}

We determine the Hausdorff and packing dimensions of Cartesian products of one-dimensional exact approximation sets. Our main result establishes the exact-approximation counterpart of the recent product theorem of Wang and Wu (2024) for limsup approximation sets, showing that passing to the substantially smaller exact approximation sets (liminf sets) does not reduce the Hausdorff dimension of the Cartesian product.  
One of the key ingredients is a refinement of the well-distributed-system framework of Bandi--Ghosh--Nandi (2023) by exploiting the fine arithmetic distribution of rational points which then gives the Hausdorff dimension of the product set under a weaker convergence condition. 

\end{abstract}
 \maketitle
\section{Exact approximation and Cartesian products}

\subsection{Introduction and the main result}
Let $\psi:\N\to(0,\infty)$ be a non-increasing function.  We write
\[
W(\psi)=\left\{x\in[0,1):\left|x-\frac pq\right|<\psi(q)
\text{ for infinitely many }(p,q)\in\Z\times\N\right\}
\]
for the usual limsup set of $\psi$-approximable numbers.  The associated exact approximation order set is defined as
\[
E(\psi)=W(\psi)\setminus\bigcup_{0<c<1}W(c\psi).
\]
Thus, $x\in E(\psi)$ if and only if $x$ is approximable at the scale $\psi(q)$, but not at any smaller scale $c\psi(q)$ with $c<1$.  Exact approximation sets are therefore finer objects than the classical limsup sets. The study of exact approximation sets was initiated by Bugeaud \cite{Bugeaud2003,Bugeaud2008, BugeaudMoreira2011}, who proved that, under mild regularity assumptions on \(\psi\), 
we have 
\begin{equation} \label{eq:gleichdim}
    \dimH E(\psi)=\dimH W(\psi).
\end{equation}
His work was motivated by earlier results of Beresnevich, Dickinson, and Velani \cite{BeresnevichDickinsonVelani}, who investigated related sets in which the upper and lower approximation rates are separated by a logarithmic gap. More recent developments include the general framework of Bandi, Ghosh, and Nandi \cite{BandiGhoshNandi}, the higher-dimensional results of Bandi and de Saxc\'e \cite{BandiDeSaxce2025}, the work of Schleischitz \cite{Joh2023RM,Joh2025MZ}, and the quantitative refinements introduced by Baker and Ward \cite{BakerWard2025}. Related but slightly larger sets  $W(\psi)\setminus\bigcap_{0<c<1}W(c\psi)$  have also been investigated recently in \cite{KLWZ}. Therein it was proven that the Hausdorff dimension of this set is equal to that of $W(\psi)$ for the case $\psi(q):=q^{-\tau}$ with $\tau>2.$ Their result also extends to arbitrary dimension.
Since the metric theory of the sets $E(\psi)\subset \R$ is rather well-understood by the work quoted above, the focus of this paper is to study Cartesian products of these sets. A basic problem in fractal geometry is to understand how Hausdorff dimension behaves under
Cartesian products. For arbitrary sets $A,B\subset\R$, Marstrand proved that
\begin{equation}\label{eq:marstrand}
\dimH(A\times B)\geq \dimH A+\dimH B,
\end{equation}
and Tricot complemented this with certain
other estimates of this type involving
packing dimension \cite{Tricot1982, Marstrand}.
While for many interesting types of sets
there is equality in \eqref{eq:marstrand},
it is well-known that strict inequality can occur.
A striking illustration goes back to Erd\"{o}s, who proved that every real number is a sum of
two Liouville numbers, and every nonzero real number is a product of two Liouville numbers
\cite{Erdos1962}. Consequently, although the set of Liouville numbers has Hausdorff dimension zero, its
Cartesian square nevertheless has Hausdorff dimension one. 

Recent work, building up to some extent on Erd\"os paper, shows that indeed
strict inequality in \eqref{eq:marstrand} is a frequent
phenomenon for sets naturally arising in Diophantine approximation. In
particular, Schleischitz studied sumsets and Cartesian products for classes determined by fixed
irrationality exponents and generalisations
to
exact approximation, while Wang and Wu
established a general product principle for limsup sets in Diophantine approximation
\cite{Joh2023RM,WangWuTAMS}.

This naturally leads to the problem of determining the Hausdorff dimension of Cartesian products of exact approximation sets, which is the primary focus of the present paper. As a byproduct we are also able to evaluate the packing dimension.

\begin{theorem}\label{thm:main}
Let $\psi_1,\ldots,\psi_m:\N\to(0,\infty)$ be non-increasing functions such that
\begin{align}\label{condithm1}
\psi_i(q)=o(q^{-2}),
\qquad i=1,\dots,m,
\end{align}
and
\begin{align}\label{condithm2}
\sum_{q=1}^{\infty} q\psi_i(q)<\infty,
\qquad i=1,\ldots,m.
\end{align}
Then
\begin{equation} \label{eq:e1}
\dimH\bigl(E(\psi_1)\times\cdots\times E(\psi_m)\bigr)
=
m-1+\min_{1\leq i\leq m}\dimH E(\psi_i).
\end{equation}
Consequently
\begin{equation} \label{eq:packingd}
  \dim_P\bigl(E(\psi_1)\times\cdots\times E(\psi_m)\bigr)
=
m.
\end{equation}
\end{theorem}



The packing dimension claim \eqref{eq:packingd} follows by applying
\eqref{eq:e1} with \(m+1\) factors and then using Tricot's product
inequality~\cite{Tricot1982}
\[
\dim_H(A\times B)\leq \dim_P(A)+\dim_H(B).
\]
Indeed, take
\[
A=E(\psi_1)\times\cdots\times E(\psi_m)
\qquad\text{and}\qquad
B=E(\psi_{m+1}),
\]
where \(\psi_{m+1}(q)=e^{-q}\). The latter decays sufficiently rapidly
that its contribution to the minimum in \eqref{eq:e1} is zero. Hence,
\eqref{eq:e1}, together with the preceding product inequality, yields
\[
\dim_P(A)
\geq
\dim_H(A\times B)-\dim_H(B),
\]
which gives \eqref{eq:packingd}.
Note that \eqref{eq:packingd} does not appear to have been known in
this generality even in the one-dimensional case \(m=1\); for sufficiently
rapidly decreasing approximation functions, a related conclusion follows
from~\cite[Corollary~7]{Joh2023RM}. Moreover, the case \(m=1\) does not by
itself yield the result for arbitrary \(m\), since packing dimension is not
additive under Cartesian products. In general, one only has Tricot's
inequality~\cite{Tricot1982}
\[
\dim_P(A\times B)\leq \dim_P(A)+\dim_P(B).
\]

Theorem~\ref{thm:main} should be compared with the recent dimension
formula of Wang and Wu~\cite{WangWuTAMS} for Cartesian products of the
larger limsup sets \(W(\psi)\). Their result implies, without assuming
either \eqref{condithm1} or \eqref{condithm2}, that
\begin{equation}\label{eq:e2}
\dimH\bigl(W(\psi_1)\times\cdots\times W(\psi_m)\bigr)
=
m-1+\min_{1\le i\le m}\dimH W(\psi_i).
\end{equation}
The theorem of Wang and Wu applies to general
limsup sets.
Given the additional
assumptions \eqref{condithm1} and \eqref{condithm2}, formula
\eqref{eq:e1} shows that replacing the limsup sets \(W(\psi_i)\) by the
smaller exact approximation sets \(E(\psi_i)\) does not alter the
Hausdorff dimension of their Cartesian product. To see this, note first that in both \eqref{eq:e1} and \eqref{eq:e2} the
upper bound is immediate, for example from Tricot's product inequalities. 
 On the other hand, by \eqref{eq:gleichdim},
which in fact does not require the convergence assumption \eqref{condithm2}
from Theorem~\ref{thm:main}, the right-hand sides of
\eqref{eq:e1} and \eqref{eq:e2} coincide, and therefore so do the
left-hand sides.

\begin{remark}
Note that, Bandi--Ghosh--Nandi~\cite{BandiGhoshNandi} developed a general framework for exact
approximation. Their
result applies to arbitrary regular metric spaces equipped with
well-distributed systems and recovers classical rational approximation on
the real line as a special case after expressing approximation in terms of
an appropriate radius function. To be precise, restricting to our classical setting of rational approximations to real numbers, their assumptions translate to $$\sum_{q\ge 1}q^3\psi(q)<\infty$$
whereas we assume \eqref{condithm2} which is clearly weaker. 

\end{remark}

Combining Theorem~\ref{thm:main} with the one-dimensional dimension formula for the exact
approximation sets due to Bugeaud and Moreira~\cite{BugeaudMoreira2011} yields an explicit corollary in terms of the
lower order at infinity. We treat $2/\infty=0$
in the claim.

\begin{corollary}\label{cor:lowerorder}
With notation as above, for $i=1,\ldots,m$, let
$\psi_i$ be decreasing and with
\begin{equation}  \label{eq:loword}
\lambda_i:=\liminf_{q\to\infty}\frac{-\log \psi_i(q)}{\log q},
\end{equation}
and assume that $\lambda_i>2$ for each $i$.
Then
\[
\dimH\bigl(E(\psi_1)\times\cdots\times E(\psi_m)\bigr)
=
m-1+\min_{1\leq i\leq m}\frac{2}{\lambda_i}, \qquad
\dim_P\bigl(E(\psi_1)\times\cdots\times E(\psi_m)\bigr)
=m.
\]
In particular, the formula holds for $\psi_i(q)=q^{-\lambda_i}$ with $\lambda_i>2$.
\end{corollary}
Again this refines the corresponding result
for products of the sets $W(\psi_i)$ from~\cite{WangWuTAMS}.
Corollary~\ref{cor:lowerorder} extends to certain collections of $\psi_i$
with some $\lambda_i=2$, subject to the convergence condition of Theorem~\ref{thm:main} being satisfied. 
We remark that the Hausdorff dimension result in special case $m=2$ and $\lambda_1=\infty, \lambda_2>(5+\sqrt{17})/2$ follows alternatively from a result on sumsets~\cite[Theorem~3.6]{Joh2023RM}; indeed, this
more generally applies in the setting of Theorem~\ref{thm:main} for pairs of approximation functions $\psi_1, \psi_2$ of this decay order, not necessarily monotonic, and provides infinite $1$-dimensional Hausdorff measure. Consequently the  packing dimension result had also been known for $m=1$ and $\lambda_1>(5+\sqrt{17})/2$, see~\cite[Theorem~3.6]{Joh2023RM}.

\subsection{On the conditions}
We briefly discuss the necessity 
of our conditions in the theorem.
Conditions~\eqref{condithm1} and  \eqref{condithm2} 
cannot be
replaced by the weaker Dirichlet upper bound \(\psi_i(q)\le q^{-2}\) or omitted altogether. 
Indeed, combining estimates of Tricot~\cite{Tricot1982} with metrical results on badly approximable numbers 
yields the following
facts.

\begin{theorem}  \label{T2}
    Let $\psi_1,\ldots,\psi_m:\N\to(0,\infty)$ be any functions.
    If \eqref{eq:packingd} holds then we must have
\begin{equation} \label{eq:nicht}
    \liminf_{q\to\infty} q^2 \psi_i(q) = 0, \qquad 1\le i\le m.
\end{equation}
 If \eqref{eq:e1} holds, then \eqref{eq:nicht} may fail for at most one index; moreover if so letting $i=m$ be this exceptional index
 and $m\ge 2$, if
 \begin{equation} \label{eq:kannsein}
     \liminf_{q\to\infty} q^2 \psi_m(q) > \frac{1}{3},
 \end{equation}
 then
 \begin{equation} \label{eq:haelt}
\dim_H E(\psi_1)= \cdots=\dim_H E(\psi_{m-1})=1,\qquad \dim_H E(\psi_m)=0, 
\end{equation}
and
\[
\dimH\bigl(E(\psi_1)\times\cdots\times E(\psi_m)\bigr)
=\dim_P\bigl(E(\psi_1)\times\cdots\times E(\psi_m)\bigr)
=
m-1.
\]
\end{theorem}

\begin{proof} In the 1-dimension setting,  it is well-known that 
\[
 \dim_P E(\psi) < 1
\]
when \eqref{eq:nicht} does not hold
for $\psi=\psi_1$. Indeed, then $E(\psi)$ consists
of numbers with uniformly bounded partial quotients 
(equivalently $c$-badly approximable for some
fixed $c>0$);
the conclusion follows from combination of~\cite{MauldinUrbanski4} asserting equality
for Hausdorff and packing dimension of the sets $F_m$
with all partial quotients $\le m$ combined with \cite{BroderickKleinbock}
where
the Hausdorff dimension of these sets $F_m$
being strictly less than $1$ was established.

Now the necessity of \eqref{eq:nicht} for \eqref{eq:packingd} follows
directly from applying Tricot's formula
\begin{equation} \label{eq:diepack}
\dim_P(\prod A_j) \le \sum \dim_P(A_j),
\end{equation}
valid for any collection of measurable sets $A_j$,
to $A_j=E(\psi_j)$, $1\le j\le m$. 

Next, for the claim on \eqref{eq:e1} we may assume $m\ge 2$. Now
on the other hand the estimate
\begin{equation} \label{eq:obenda}
\dim_H(A\times B)\le \dim_H(A) + \dim_P(B)
\end{equation}
for measurable $A,B$ again due to Tricot implies
\[
\dimH\bigl(E(\psi_1)\times\cdots\times E(\psi_m)\bigr)
\le \dim_P\bigl(E(\psi_1)\times\cdots\times E(\psi_{m-1})\bigr) + \dim_H(E(\psi_m)).
\]
Hence, if without loss of generality 
$m$ is the index inducing the minimum 
in \eqref{eq:e1}, then this requires 
the packing dimension expression in the right hand side
to be $m-1$, which we have seen is impossible 
if \eqref{eq:nicht} fails for some $i=1,2,\ldots,m-1$.
Hence \eqref{eq:nicht} may only possibly fail for $i=m$.

For the last claim,  note that 
if \eqref{eq:kannsein} holds then $E(\psi_m)$ is at most countable, see for example
\cite{Bugeaud2008}. By stability of Hausdorff dimension under countable unions, we get
\[
\dimH\!\left(
E(\psi_1)\times\cdots\times E(\psi_m)
\right)
\le \dimH\!\left(
E(\psi_1)\times\cdots\times E(\psi_{m-1})
\right)
\le m-1.
\]
Since the right-hand side of~\eqref{eq:e1} is at least \(m-1\), formula \eqref{eq:e1} can only hold if we have equality in the right inequality above.
But by \eqref{eq:obenda}, this implies the left identities in \eqref{eq:haelt}, the right identity holds as the set is countable. Then 
further the product must have Hausdorff dimension 
precisely $m-1$. Since $E(\psi_m)$ is countable
it readily follows from \eqref{eq:diepack} that the packing dimension must be $m-1$ as well.
\end{proof}

\begin{remark}
    The product sets we study may even be empty if \eqref{eq:nicht} fails
    for some index.
    For example $E(0.4\cdot q^{-2})=\emptyset$, thus the implications \eqref{eq:e1} and \eqref{eq:packingd} of the theorem fail trivially.
    Clearly in Theorem~\ref{thm:main} this cannot happen
    due to condition \eqref{condithm1} which implies \eqref{eq:nicht}.
\end{remark}

   
    The claims in Theorem~\ref{T2} are specific to exact approximation sets and do not apply to the setting of limsup sets considered in \cite{WangWuTAMS}.
   It is worth pointing out that there is still
    a gap between the sufficient conditions \eqref{condithm1} and  \eqref{condithm2} on the one hand
    and the (essentially) necessary condition \eqref{eq:nicht} on the other hand.
    Indeed, it would be nice to identify the weakest assumptions under which the claims \eqref{eq:e1} and \eqref{eq:packingd} hold.  
   

\subsection{Outline of the proof}
We briefly outline the main idea behind our construction. The principal
novelty is a new method for counting rational approximants that contribute
to exact approximation in one dimension. This provides a selection
criterion for the rationals relevant to exact approximation and refines
the approach developed in~\cite[Section~2]{BandiGhoshNandi}. We then
extend this counting method to Cartesian products, leading to a
considerably simpler framework than that introduced by Wang and
Wu~\cite{WangWuTAMS}.

Unlike in the limsup setting considered
in~\cite{WangWuTAMS}, the coordinates can no longer be treated
independently. In the exact approximation (liminf) setting, achieving
exact approximation in one coordinate must be accompanied by the
simultaneous exclusion of excessively good rational approximations in the
remaining coordinates. It is precisely this interplay between
approximation and avoidance that underlies our construction.

We summarise below some notations that will be used throughout this paper.
\begin{itemize}
    \item For a Lebesgue measurable set \(E\subset\mathbb R\), we denote its Lebesgue measure by
    \(\mathcal L(E)\). For brevity, we occasionally write \(|E|\) instead of \(\mathcal L(E)\).

    \item For a Lebesgue measurable set \(E\subset\mathbb R^m\), \(m\ge2\), we denote its \(m\)-dimensional Lebesgue measure by \(\mathcal L^m(E)\).

    \item For a cube \(B\subset\mathbb R^m\), we denote by \(\gamma(B)\) its half side length.

    \item For an interval \(U\subset\mathbb R\) and \(r>0\), we define the dilation of \(U\) by
    \[
    rU
    =
    \left\{
    x\in\mathbb R:
    |x-u_0|
    \le
    \frac{r|U|}{2}
    \right\},
    \]
    where \(u_0\) denotes the midpoint of \(U\).

    \item For integers \(p,q\), \(\gcd(p,q)\) denotes their greatest common divisor.

    \item A collection of balls is said to be \emph{\(c\)-separated} if the distance between the centres of any two distinct balls is at least \(c\).
\end{itemize}

\section{Auxiliary lemmas}\label{pre}
First, we introduce two auxiliary lemmas that form the basis of our
approach to exact approximation in one dimension. The first identifies a
large collection of rationals that are relevant for
\(\psi\)-approximation. However, some of these rationals need not be
sufficiently separated at the finer scale \(c\psi_i\), and therefore
cannot be used directly in the exact approximation setting. To overcome
this difficulty, we apply Lemma~\ref{count} to extract a subcollection
enjoying a strong separation~property.

\subsection{Distribution of rationals}\label{sec21}
The distribution of rational numbers with bounded denominators plays a
fundamental role in metric Diophantine approximation and underpins the
ubiquity framework developed in
\cite{BDV,DRV2, WangWu19}.
In particular, rational points exhibit both strong separation properties
and local ubiquity, which are the two features required in our
construction. We recall below the form of these properties needed for the
present paper.

Fix a large integer \(N\ge 8\), and let
\(Q\ge N\) be sufficiently large (the precise lower bound on \(Q\),
depending on the ambient interval, will be specified later). We consider
the collection
\[
\mathcal Q_Q
=
\left\{
\frac pq:
QN^{-1}\le q<Q,\;
\gcd(p,q)=1
\right\}.
\]
It follows from the coprimality of the rationals that the gap between any two distinct points in $\mathcal{Q}_{Q}$ admit a uniform positive lower bound.
\begin{lem}\label{WDP}For any interval $U$, let $Q$ be sufficiently large such that
        \begin{equation}\label{f1}
        Q^2\cdot|U|\ge 16,\qquad
       {8Q^{-1}\log (QN^{-1})}\le \frac{1}{16}|U|.
        \end{equation} Then, there exists a subset $\mathcal{A}_{Q}(U)\subset \mathcal{Q}_{Q}$ such that
\begin{itemize}
  \item For any two distinct rationals $p/q$ and $p^{\prime}/q^{\prime}$ in $\mathcal{Q}_Q$, one has
   $$\left|\frac{p}{q}-\frac{p^{\prime}}{q^{\prime}}\right|\ge \frac{1}{Q^2}.$$
     \item The union of the following disjoint balls is  contained in $U$
        $$\bigcup_{\frac{p}{q}\in \mathcal{A}_{Q}(U)}B\left(\frac{p}{q},\frac{1}{2Q^2}\right)\subset U.$$
    \item The cardinality of \(\mathcal A_Q(U)\) satisfies
        \begin{align}\label{number}
        \# \mathcal{A}_{Q}(U)\asymp_{N} \mathcal{L}(U)\cdot Q^{2}.
        \end{align}
\end{itemize}
\end{lem}
Whenever there is no ambiguity, we simply write $\mathcal{A}_{Q}$ instead of $\mathcal{A}_{Q}(U)$ for the ease of notation. For related distribution results for rational points in a more abstract
setting, we refer the reader to
\cite[Definition~1.3]{BandiGhoshNandi}.

    \begin{proof}
        The first item is clear from the definition
        of $\mathcal{Q}_{Q}$.
        We note that the first inequality in (\ref{f1}) ensures that the set $\mathcal{Q}_{Q}\cap \frac12U$ is non-empty.
        Let $\mathcal{A}_{Q}(U)$ be the collection of reduced rationals in $\mathcal{Q}_Q \cap \frac{1}{2}U$. We claim that it has the desired properties. It is clear that $$
        \bigcup_{\frac{p}{q}\in \mathcal{A}_{Q}(U)}B\left(\frac{p}{q},\frac{1}{2Q^2}\right)\subset U.
        $$
Now we count the number of the elements in $\mathcal{A}_{Q}(U)$. By a volume argument, we immediately obtain the upper bound
$$\#\mathcal{A}_{Q}(U)\cdot Q^{-2}\le |U|\Longrightarrow \#\mathcal{A}_{Q}(U)\le |U|\cdot Q^{2}.$$
Next, by letting
$$\widetilde{U}:=\bigcup_{1\le q< QN^{-1}}\left\{x\in \frac12U:\left|x-\frac{p}{q}\right|<\frac{1}{qQ}, \ {\text{for some integer}}\ p\right\},$$
we claim that
\begin{align}\label{include}
   (1/4\cdot U)\setminus \widetilde{U}\subset \bigcup_{\frac{p}{q}\in \mathcal{A}_{Q}(U)}B\left(\frac{p}{q},\frac{1}{Q^{2}N^{-1}}\right).
   \end{align}
In fact, recalling Dirichlet's theorem that for any $x\in \frac14U$, there exists a reduced rational point $p/q$ such that
        $$\left|x-\frac{p}{q}\right|<\frac{1}{qQ}\textrm{ and }1\le q\le Q.$$
        When $x\in \widetilde{U}^c$ this forces the denominator to lie in $QN^{-1} \le q<Q$. Together with the definition of $\mathcal{A}_{Q}(U)$, the claim is clear.

        Therefore, the lower bound of $\#\mathcal{A}_{Q}(U)$ is obtained by estimating the Lebesgue measure of
        $$(1/4\cdot U)\setminus \widetilde{U}.$$
        We note that the interval $\frac14U$ contains at most
        $$\max\big\{|1/4\cdot U|\cdot q,2\big\}\le q\cdot |1/4\cdot U|+2$$
        rationals with denominator $q$, by the $1/q$-separability of these rationals.
        Hence, the Lebesgue measure of the set $\widetilde{U}$ is less than
        $$\sum_{q=1}^{QN^{-1}-1}\frac{2}{qQ}\bigl(q\cdot|1/4\cdot U|+2\bigr)
\le \frac{1}{2N}|U|
   +\frac{8\log (QN^{-1})}{Q}
\le \frac{1}{8}|U|,$$
where the last inequality follows from the inequality (\ref{f1}) and $N\ge 8$. It implies that
        $$|(1/4\cdot U)\setminus \widetilde{U}|\ge \frac18 |U|.$$

        Therefore, by (\ref{include}), we conclude that
        $$\frac{1}{8}|U|\le\#\mathcal{A}_{Q}(U)\cdot 2Q^{-2}N\Longrightarrow \#\mathcal{A}_{Q}(U)\ge \frac{1}{16}|U|\cdot Q^{2}N^{-1}.$$
        We have finished the proof of the second and third items.

    \end{proof}
{
\subsection{Selection criterion}\label{select}
Let $\psi:\mathbb{N}\rightarrow \mathbb{R}^+$ be a positive function satisfying the assumptions imposed on $\psi_1$ in Theorem \ref{thm:main}. 
For each rational $p/q$, for any $0<c<1$, we define the annulus (union of two intervals)
$$A_c(p/q)=B\left(\frac{p}{q},\psi(q)\right)\setminus B\left(\frac{p}{q},c\psi(q)\right).$$ 

Fix a rational point $p_0/q_0$ with $\textrm{gcd}(p_0,q_0)=1$. For any $0<c<1$, choose an arbitrary interval in the annulus $A_{c}(p_0/q_0)$ and denote it by $U_0$. Let $N\ge 8$ and $Q_-$ be large integers. For any sufficiently large integer $Q_+>Q_-$ such that $Q_+-Q_->N$ and (\ref{f1}) holds with $U=U_0$, we apply Lemma \ref{WDP} by taking parameter $Q=Q_+$, which shows the existence of a subset $\mathcal{A}_{Q_+}\subset \mathcal{Q}_{Q_+}$ satisfying the following properties:
\begin{itemize}
  \item The distance between any two distinct rationals in $\mathcal{A}_{Q_+}$ is at least $Q_+^{-2}$.

     \item The union of the disjoint balls centered in $\mathcal{A}_{Q_+}$ with radius $\frac12Q_+^2$ are contained in $U_0$.
    \item The number of these balls is asymptotically $\mathcal{L}(U_0)\cdot Q_+^{2}$.
\end{itemize}


Our criterion for selecting rationals in $\mathcal{A}_{Q_+}$ is as follows: we retain those rationals whose associated annulus of width $(1-c)\psi(q)$ is disjoint from every ball of radius $c\psi(q)$ centred at a rational point with denominator $q$ in the range $Q_-\le q< Q_+$. By this choice, it follows that the retained rationals automatically avoid rationals whose denominators lie in the interval $[Q_-,Q_+)$.

Now, we estimate the number of rationals eliminated from $\mathcal{A}_{Q_+}$, that is, calculate the cardinality of set
\begin{align*}
\mathcal{D}_{Q_+}(c) := \Bigg\{ \frac{p'}{q'} \in \mathcal{A}_{Q_+} :\; & A_{c}(p'/q') \cap B\left(\frac{p}{q}, c\psi(q)\right) \neq \emptyset, \\
& \quad \text{for some rational point } \frac{p}{q} \text{ with } Q_-\le q < Q_+ \Bigg\}.
\end{align*}
The strategy is to split $[Q_-,Q_+)$ into finitely many pieces of the form
$$[N^{-k_-}Q_+,N^{-k_-+1}Q_+),\ \cdots [N^{-2}Q_+,N^{-1}Q_+),\ [N^{-1}Q_+,Q_+),$$
where $k_-$ is the smallest integer satisfy that $N^{-k_-}Q_+\le Q_-$, and then sum the counts obtained from each piece. In addition, we can further require that $Q_+$ satisfies
 \begin{align}\label{lcondi}
    \sum_{1\le t\le k_-}N^{-2t}Q_+^2\psi(N^{-t}Q_+^2)\le \frac{1}{1152}(1-c)N^{-7}.
    \end{align}
     Indeed, together with the monotonicity of $\psi$ and by (\ref{condithm2}) which is the convergence of the series 
     $\sum_{q\ge 1}q\psi(q),$ 
     the Cauchy condensation test allows us to choose $Q_-$ and $Q_+$ such that (\ref{lcondi}) holds. Since $c$ is fixed once $Q_+$ is chosen, we simply write $\mathcal{D}_{Q_+}$ instead of $\mathcal{D}_{Q_+}(c)$ for the ease of notation.  We prove the following counting lemma.
\begin{lem}[Counting Lemma]\label{count}
    Under the notations as above, we have $\#\mathcal{D}_{Q_+}\le \#\mathcal{A}_{Q_+} / 2$.
\end{lem}
\begin{proof}\label{trequire}
    It readily follows, by the assumption {(\ref{condithm1})} on $\psi$, that
    \begin{align}\label{psicondi}
    \psi(q)<\frac{1}{2N^2}\cdot\frac{1}{q^2}\text{ for large }q.
    \end{align}
     We partition $\mathcal{D}_{Q_+}$ into finitely many subsets as follows. For any $1\le t\le k_-$, let
     \begin{align*}
\mathcal{D}_{Q_+,t}:= \Bigg\{ \frac{p'}{q'} \in \mathcal{A}_{Q_+} : & A_{c}(p'/q') \cap B\left(\frac{p}{q}, c\psi(q)\right) \neq \emptyset, \\
& \quad \text{for some rational point } \frac{p}{q} \text{ with } N^{-t}Q_+\le q< N^{-t+1}Q_+ \Bigg\}.
\end{align*}
Now, we use a volume argument to estimate its cardinality. For any rational $p^{\prime}/q^{\prime}\in\mathcal{D}_{Q_+,t}$, there exists a rational point $p/q$ such that $$ N^{-t}Q_+\le q<N^{-t+1}Q_+ \textrm{ and } \left|\frac{p}{q}-\frac{p^{\prime}}{q^{\prime}}\right|\le c\psi(q)+\psi(q^{\prime}).$$
We claim that ${p}/{q}\neq{p^{\prime}}/{q^{\prime}}.$ Otherwise, since $\textrm{gcd}(p^{\prime},q^{\prime})=1$, one has $q=rq^{\prime}$ for some integer $r\ge 1$. It follows that $q\ge q^{\prime}$. By the non-increasing of $\psi$, one has $c\psi(q)\le c\psi(q^{\prime})$. Therefore, by the definition of annulus, it follows that
$$A_{c}(p'/q') \cap B\left(\frac{p^{\prime}}{q^{\prime}}, c\psi(q^{\prime})\right) =\emptyset\ \Longrightarrow\  A_{c}(p'/q') \cap B\left(\frac{p}{q}, c\psi(q)\right) =\emptyset,$$
which yields a contradiction. On the other hand, for any two distinct rational points, one has
\begin{align*}
\left|\frac{p}{q}-\frac{p^{\prime}}{q^{\prime}}\right|\ge\frac{1}{qq^{\prime}}.
\end{align*}
By (\ref{psicondi}) applied to $q^{\prime}$, it follows that
$$\frac{1}{qq^{\prime}}\le  \left|\frac{p}{q}-\frac{p^{\prime}}{q^{\prime}}\right|\le  c\psi(q)+\psi(q^{\prime})\le c\psi(q)+\frac{1}{2N^2q^{\prime2}}.$$
Together with the range of $q$ and $q^{\prime}$ and $t\ge 1$, one has $$qq^{\prime}< N^{-t+1}Q_+\cdot Q_+\le Q_{+}^2
= N^2\cdot (N^{-1}Q_+)^2\le N^2q^{\prime2}, $$
where the second inequality holds since $t\ge 1$.
    Combining the last two properties, we conclude that
    $$\frac{1}{2N^2}\cdot\frac{1}{q^{\prime2}}\le
    \frac{1}{qq^{\prime}}-\frac{1}{2N^2q^{\prime2}}
    \le c\psi(q),$$
    and consequently
    \begin{align*}
    \left|\frac{p}{q}-\frac{p^{\prime}}{q^{\prime}}\right|
    \le c\psi(q)+\psi(q^{\prime})
    \le 
    c\psi(q)+\frac{1}{2N^2q^{\prime2}}
    \le 2c\psi(q).
    \end{align*}
    Hence, using again $1/(2N^2q^{\prime2})\le c\psi(q)$, for any $p^{\prime}/q^{\prime}\in\mathcal{D}_{Q_+,t}$, we have found a rational point $p/q$ with $N^{-t} Q_+\le q<N^{-t+1}Q_+$ such that
    $$B\left(\frac{p^{\prime}}{q^{\prime}},\frac{1}{2N^2q^{\prime2}}\right)\subset
    B\left(\frac{p^{\prime}}{q^{\prime}},c\psi(q)\right)
    \subset B\left(\frac{p}{q},3c\psi(q)\right).$$
    Thus, we obtain the inclusion
    \begin{align}\label{result1}
       \bigcup_{p^{\prime}/q^{\prime}\in\mathcal{D}_{Q_+}}B\left(\frac{p^{\prime}}{q^{\prime}},\frac{1}{2N^2q^{\prime2}}\right)\subset \bigcup_{1\le t\le k_-}\bigcup_{\substack{p/q:N^{-t} Q_+\le q<N^{-t+1}Q_+}}B\left(\frac{p}{q},3c\psi(q)\right),
    \end{align}
    where the balls in the left-hand side are pairwise disjoint by the separation of rationals in $\mathcal{D}_{Q_+,t}$. We note that the second union of balls in the right-hand side should be further restricted to
    $$B\left(\frac{p}{q},c\psi(q)\right)\cap U_0\neq\emptyset.$$
    Thus, it leads us to estimate the cardinality of following set
    $$\mathcal{N}_{Q_+,t}=\left\{\frac{p}{q}:B\left(\frac{p}{q},c\psi(q)\right)\cap U_0\neq\emptyset, \ N^{-t}Q_+\le q< N^{-t+1}Q_+\right\}.$$
    By similar argument as above, for any $p/q\in\mathcal{N}_{Q_+,t}$, we have $p/q\neq p_0/q_0.$ Similarly, if $p/q= p_0/q_0$, then $q\ge q_0$ by $\textrm{gcd}(p_0,q_0)=1$. Since $U_0$ is one of the interval of $A_{c}(p_0/q_0)$, it follows that
    $$B\left(\frac{p_0}{q_0},c\psi(q_0)\right)\cap U_0=\emptyset\ \Longrightarrow\ B\left(\frac{p}{q},c\psi(q)\right)\cap U_0=\emptyset.$$
    Therefore, we have that
    $$\frac{1}{q^{2}}\le\frac{1}{q_0q}\le \left|\frac{p_0}{q_0}-\frac{p}{q}\right|\le c\psi(q)+\psi(q_0)<\frac{1}{2q^{2}}+\psi(q_0)\Longrightarrow \frac{1}{2q^{2}}\le \psi(q_0),$$
    where the fourth inequality holds by (\ref{psicondi}) with $N\ge 8$ and $c<1$.
    It implies that
    $$\left|\frac{p_0}{q_0}-\frac{p}{q}\right|\le 
    \frac{1}{2q^{2}}+\psi(q_0)\le 2\psi(q_0).$$
    Then, for any rational point $p/q\in\mathcal{N}_{Q_+,t}$, we have that
    $$B\left(\frac{p}{q},\frac{1}{2N^2q^2}\right)\subset B\left(\frac{p_0}{q_0},3\psi(q_0)\right).$$
    Therefore, we have the following inclusion
    $$\bigcup_{{p}/{q}\in\mathcal{N}_{Q_+,t}}B\left(\frac{p}{q},\frac{1}{2N^2q^{2}}\right)\subset B\left(\frac{p_0}{q_0},3\psi(q_0)\right).$$
    where the balls in the left-hand side are pairwise disjoint by separation of rationals in $\mathcal{N}_{Q_+,t}$. By a volume argument, one has
    $$\#\mathcal{N}_{Q_+,t}\cdot \frac{1}{2N^2\cdot (N^{-t+1}Q_+)^2}\le 3\psi(q_0).$$
    Returning back to (\ref{result1}), by a volume argument again, one has
    \begin{align*}
         \#\mathcal{D}_{Q_+}\cdot\frac{1}{2N^2\cdot Q_+^2}&\le \sum_{1\le t\le k_-}\#\mathcal{N}_{Q_+,t}\cdot 3c\psi(N^{-t}Q_+)\\
         &\le \sum_{1\le t\le k_-}{2N^{-2t+4}Q_+^2}\cdot3\psi(q_0)\cdot3\psi(N^{-t}Q_+).
    \end{align*}
    Thus, we conclude that
    \begin{align*}
    \#\mathcal{D}_{Q_+}&\le{36 \cdot Q_+^2 N^6}\psi(q_0)\cdot\sum_{1\le t\le k_-}N^{-2t}Q_+^2\psi(N^{-t}Q_+)\\
    &\le\#\mathcal{A}_{Q_+}\cdot 576(1-c)^{-1}N^7\sum_{1\le t\le k_-}N^{-2t}Q_+^2\psi(N^{-t}Q_+)\\
    &\le \#\mathcal{A}_{Q_+}/2,
    \end{align*}
    where the second inequality holds by the lower bound estimate of the third item in Lemma \ref{WDP}, the last inequality holds by the choice of $Q_-$ and $Q_+$ in (\ref{lcondi}).
\end{proof}

Therefore, after the selection process,  it retains a considerable number of rational points, which is exactly what we need for exact approximation in one-dimensional case. We associate an annulus to each such point, choose an arbitrary interval in the annulus, repeat the above process in that interval and select related parameters appropriately. This inductive process eventually achieves the desired Cantor subset of $E(\psi)$. As the main purpose of this paper is the Cartesian product, we would not go into detailed discussion of these steps.

Finally, towards the end of this section, we introduce some auxiliary lemmas. The first is the cutting lemma which partitions the rectangle into well-separated cubes.
\begin{lem}[Cutting lemma \cite{WangWuTAMS}]\label{cut}
    Let $$R=B(x_1,r_1)\times \cdots\times B(x_m,r_m)$$
    be a rectangle in $[0,1]^m$. Then inside $R$, there is a collection $\mathcal{C}$ of $5r$-separated balls
    with the same radius $r=\min\{r_i:1\le i\le m\}$ satisfying
    $$\sum_{B\in\mathcal{C}}\mathcal{L}^m(B)\asymp\mathcal{L}^m(R).$$
\end{lem}
The following lemma is a basic tool in estimating the Hausdorff dimension. 
\begin{lem}[Mass distribution principle \cite{Falconer_book2013}]\label{MDP}
Let $E\in[0,1)$ be a Borel set and $\mu$ be a measure with $\mu(E)>0$. If for any $x\in E$,
$$\liminf_{r\rightarrow 0}\frac{\log \mu(B(x,r))}{\log r}\ge s,$$
then $\dimH E\ge s$.
\end{lem}


\section{Cantor subset}\label{sec3}
In this section, we will construct a reasonable subset of $E(\psi_1)\times\cdots\times E(\psi_m)$, whose Hausdorff dimension attains the claimed lower bound of the theorem. To present more clearly how the one-dimensional construction extends to the Cartesian product setting, we will frequently use some notations introduced in Section~\ref{pre}.

For each $1\le i\le m$, let $\{\widetilde{Q}_n^{(i)}\}_{n\ge 1}$ be a sequence of integers such that
$$\lim_{n\rightarrow\infty}-\frac{\log \psi_i(\widetilde{Q}_n^{(i)})}{\log \widetilde{Q}_n^{(i)}}=\liminf_{q\rightarrow\infty}-\frac{\log \psi_i(q)}{\log q}:=\alpha_i.$$
Note that the proof below is also valid for the case of certain $\alpha_i=\infty$.
Let $\{c_{n}\}_{n\ge 1}$ be a sequence of increasing real numbers such that $c_{n}\rightarrow 1\text{ as }n\rightarrow \infty$ and $N$ be a sufficiently large integer. We choose sufficiently large integers $Q_{1-m},Q_{2-m}\ldots,Q_{0}$, which will serve as the initial lower bounds for the denominators of rationals to be avoided in the respective coordinate directions.
For any integer $p_{1-m}$ such that $\textrm{gcd}(p_{1-m},Q_{1-m})=1$, choose an arbitrary interval in the annulus $A_{c_1}(p_{1-m}/Q_{1-m})$ and denote it by $U_0$. We then take the starting ball as $B_0:=U_0^m$, i.e., the Cartesian product of $m$ copies of $U_0$.
Moreover, we denote the $i$-th coordinate
of the set $E\subset \mathbb{R}^m$ by $E\vert_i$ in the below sections.

Now, we begin to construct the $1$-st level of the Cantor set. At this level, we focus solely on approaching exact approximation in the first coordinate.
\begin{enumerate}
    \item {Dividing $B_0\vert_1$}: Applying Lemma \ref{WDP} to $B_0\vert_1$, there is $\widehat{Q}_1$ such that (\ref{f1}) holds for any $Q\ge \widehat{Q}_1$. Set $Q_-=Q_{1-m}$ and $\psi=\psi_1$ as in Section \ref{select}, and choose
$$Q_+\ge \widehat{Q}_1,\ Q_+\in\{\widetilde{Q}_n^{(1)}\}_{n\ge 1}$$
satisfying (\ref{lcondi}); define $Q_1=Q_+$. Then we get a collection of rationals $\mathcal{A}_{Q_1}(B_{0}\vert_1)$ satisfying the following properties:
        \begin{itemize}
  \item The gap between any two distinct rationals in the collection is at least $Q_1^{-2}$.
     \item The union of balls centered in the above collection with radius $\frac12Q_1^2$ are contained in $B_{0}\vert_1$.
    \item The number of these balls is asymptotically $\mathcal{L}(B_{0}\vert_1)\cdot Q_1^2$.
\end{itemize}
 These rational points almost form a partition of $B_0\vert_1$, i.e., the balls $$B\left(\frac{p}{q},\frac{1}{2Q_1^2}\right),\ \frac{p}{q}\in \mathcal{A}_{Q_1}(B_{0}\vert_1)$$
 are pairwise disjoint and contained in $B_{0}\vert_1$.
    \item {Select appropriate rationals}: We get rid of rationals in $\mathcal{A}_{Q_1}(B_{0}\vert_1)$ whose annulus intersects $B(p/q,c_1\psi(q))$ for some reduced rational point with denominator $q$ in the range $q_0\le q< Q_1$. By Lemma \ref{count}, the number of remaining rationals is asymptotically at least
    \begin{align}\label{ranumber}
   \mathcal{L}(B_{0}\vert_1)\cdot Q_1^2.
\end{align}
    Denote these rationals by $\overline{\mathcal{Q}}_{Q_1}(B_{0}\vert_1)$. The collection of balls
$$B\left(\frac{p}{q},\frac{1}{2Q_1^2}\right),\ \frac{p}{q}\in \overline{\mathcal{Q}}_{Q_1}(B_{0}\vert_1)$$
is denoted by $\mathcal{B}(B_0\vert_1)$. 
    \item Shrink to form annuli: Recall (\ref{psicondi}), together with the fact that the rationals in $\overline{\mathcal{Q}}_{Q_1}(B_0\vert_1)$ have denominator $q\ge Q_1N^{-1}$, we have
    $$\psi(q)<\frac{1}{2Q_1^2}.$$
    Therefore, one can shrink each ball in $\mathcal{B}(B_0\vert_1)$ so that its radius is reduced from
     $$1/2\cdot Q_1^2\textrm{ to }\psi(q).$$
    Then we form the corresponding annulus $A_{c_2}(p/q)$ from each shrunken ball. (We take $A_{c_2}(p/q)$ instead of $A_{c_1}(p/q)$ to count the set $\mathcal{N}_{Q_+,t}$ in Lemma \ref{count}.) These annuli are not only pairwise disjoint, but also disjoint from every ball of radius $c_1\psi(q)$ centered at rationals with denominator $Q_{1-m}\le q<Q_1$. Each annulus $A_{c_2}(p/q)$ is composed of two intervals. We select one of its two intervals arbitrarily and denote by $U_{c_2}(p/q)$. The collection of all such selected intervals is denoted by $\mathcal{U}_{Q_1}(B_0\vert_1)$.
    \item {Lifted to rectangles:} We lift each interval $U_{c_2}(p/q)$ in $\mathcal{U}_{Q_1}(B_0\vert_1)$ to $\R^m$ by taking its Cartesian product with $\prod_{j=2}^{m}B_0\vert_j$. 
    More precisely, we consider the following rectangles  $$R_{1}:=U_{c_2}(p/q)\times B_{0,2}\times\cdots\times B_{0,m}$$
    with $p/q\in \mathcal{Q}_{Q_1}(B_{0}\vert_1)$. Denote by $\mathcal{R}(B_0)$ the collection of these rectangles.
    \item Cut the rectangles: For each rectangle $R_1$ in $\mathcal{R}(B_0)$, we use Lemma \ref{cut} to cut the rectangle into balls, and then we get a collection $\mathcal{C}(R_{1})$ of well-separated balls of the form
    $$B_{1}:=U_{c_2}(p/q)\times I((1-c_2)\psi_1(q))\times\cdots I((1-c_2)\psi_1(q)),$$
    where $p/q\in \mathcal{Q}_{Q_1}(B_{0}\vert_1)$ and $I((1-c_2)\psi_1(q))$ denote an interval of length $(1-c_2)\psi_1(q)$. Moreover, these balls are contained in $R_{1}$, and the number of balls in $\mathcal{C}(R_{1})$ is asymptotically at least
    \begin{align}\label{est1}
        \prod_{j=2}^{m}\frac{\mathcal{L}(B_{0}\vert_j)}{(1-c_2)\psi_1(q)}.
        \end{align}

\end{enumerate}

By construction as above, the first level is
\begin{align*}
\mathcal{E}_1=\bigcup_{R_{1}\in\mathcal{R}(B_0)}\bigcup_{B_{1}\in\mathcal{C}(R_{1})}B_1.
\end{align*}

The remaining levels are now constructed inductively. Assume that the $(n-1)$-th level $\mathcal{E}_{n-1}$ has already been defined. To construct the $n$-th level $\mathcal{E}_n$, we write $n=tm+i$ with $1\le i\le m$ and $t\in \mathbb{N}$. Then, for every ball $B_{n-1}$ belonging to $\mathcal{E}_{n-1}$, we carry out the following analogous construction. At this level, we focus solely on approaching exact approximation in the $i$-th coordinate.

\begin{enumerate}
    \item {Divide $B_{n-1}\vert_i$}:
    Apply Lemma \ref{WDP} to $B_{n-1}\vert_i$, there is $\widehat{Q}_n$ such that (\ref{f1}) holds for any $Q\ge \widehat{Q}_n$. Set $Q_-=Q_{n-m}$ and $\psi=\psi_i$ as in Section 2.2, and choose
    \begin{align}\label{choiceQn}
Q_+\ge \max\{\widehat{Q}_n,e^{Q_{n-1}}\},\ Q_+\in\{\widetilde{Q}_n^{(i)}\}_{n\ge 1}
\end{align}
satisfying (\ref{lcondi}); define $Q_n=Q_+$. Then we obtain a collection of rationals $\mathcal{A}_{Q_n}(B_{n-1}\vert_i)$ such that
        \begin{itemize}
  \item The gap between any two distinct rationals in the collection is at least $Q_n^{-2}$.
     \item The union of balls centred in above collection with radius $\frac{1}{2}Q_n^{-2}$ are contained in $B_{n-1}\vert_i$.
    \item The number of these balls is asymptotically $\mathcal{L}(B_{n-1}\vert_i)\cdot Q_n^{2}$.
\end{itemize}
 Therefore, the balls $$B\left(\frac{p}{q},\frac{1}{2Q_n^2}\right),\ \frac{p}{q}\in \mathcal{A}_{Q_n}(B_{n-1}\vert_i)$$
 are pairwise disjoint and contained in $B_{n-1}\vert_i$.
    \item {Selected appropriate rationals}: We get rid of rationals in $\mathcal{A}_{Q_n}(B_{n-1}\vert_i)$ whose annulus intersects ball $B(p/q,c_n\psi(q))$ for some reduced rational point with denominator $q$ in the range $$Q_{n-m}\le q< Q_n.$$ 
        By Lemma \ref{count}, the number of remaining rationals is asymptotically at least
    \begin{align*}
   \mathcal{L}(B_{n-1}\vert_i)\cdot Q_n^2.
\end{align*}
    Denote by these rationals $\overline{\mathcal{Q}}_{Q_n}(B_{n-1}\vert_i)$. The collection of balls
$$B\left(\frac{p}{q},\frac{1}{2Q_n^2}\right),\ \frac{p}{q}\in \overline{\mathcal{Q}}_{Q_n}(B_{n-1}\vert_i)$$
is denoted by $\mathcal{B}(B_{n-1}\vert_i)$.
    \item Shrink to form annuli: Similarly, recall (\ref{psicondi}), together with the fact that the rationals in $\mathcal{Q}_{Q_n}(B_{n-1}\vert_i)$ have denominator $q\ge Q_n$, we have
    $$\psi(q)<\frac{1}{2 Q_n^2}.$$
    Therefore, one can shrink each ball in $\mathcal{B}(B_{n-1}\vert_i)$ so that its radius is reduced from
     $$1/2\cdot Q_n^{-2}\textrm{ to }\psi(q).$$
    and then form the corresponding annulus $A_{c_{t+1}}(p/q)$ from each shrunken ball. 
    We select one of its two intervals arbitrarily and denote by $U_{c_{t+1}}(p/q)$. The collection of all such selected intervals is denoted by $\mathcal{U}_{Q_n}(B_{n-1}\vert_i)$.
    \item {Lifting to rectangles:} For each interval $U_{c_{t+1}}(p/q)$ in $\mathcal{U}_{Q_n}(B_{n-1}\vert_i)$, we lift it in the $i$-th coordinate to $\R^m$ by taking its Cartesian product with $$\prod_{\substack{j=1,j\neq i}}^{m}B_{n-1}\vert_j.$$ More precisely, we replace the $i$-th coordinate of $B_{n-1}$ by $U_{c_{t+1}}(p/q)$, obtaining the rectangle $$R_{n}:=\underbrace{B_{n-1}\vert_1\times\cdots\times B_{n-1}\vert_{i-1}}_{i-1\text{ terms}} \times U_{c_{t+1}}(p/q)\times \underbrace{B_{n-1}\vert_{i+1}\times\cdots\times B_{n-1}\vert_m}_{m-i\text{ terms}}$$
    with $p/q\in \mathcal{Q}_{Q_n}(B_{n-1}\vert_i)$. Denote by $\mathcal{R}(B_{n-1})$ the collection of these rectangles.
    \item Cut the rectangles: Similarly, for each rectangle $R_n$ in $\mathcal{R}(B_{n-1})$, we use Lemma \ref{cut} to cut the rectangle into balls, and then we get a collection $\mathcal{C}(R_{n})$ of well-separated balls of the form
    $$B_{n}:=I((1-c_{t+1})\psi_i(q))^{i-1}\times U_{c_{t+1}}(p/q)\times I((1-c_{t+1})\psi_i(q))^{m-i},$$
    where $p/q\in \mathcal{Q}_{Q_n}(B_{n-1}\vert_i)$. Moreover, these balls are contained in $R_{n}$, and the number of balls in $\mathcal{C}(R_{n})$ is asymptotically 
    \begin{align}\label{est2}
        \prod_{j=1,j\neq i}^{m}\frac{\mathcal{L}(B_{n-1}\vert_j)}{(1-c_{t+1})\psi_i(q)}.
        \end{align}

\end{enumerate}

Therefore, the $n$-th level is
\begin{align*}
\mathcal{E}_n=\bigcup_{B_{n-1}\in\mathcal{E}_{n-1}}\bigcup_{R_{n}\in\mathcal{R}(B_{n-1})}\bigcup_{B_{n}\in\mathcal{C}(R_{n})}B_n.
\end{align*}
It follows that the Cantor subset is defined as
$$\mathcal{E}_{\infty}=\bigcap_{n=1}^{\infty}\mathcal{E}_n.$$

At the end of this section, we give an explanation on $$\mathcal{E}_{\infty}\subset E(\psi_1)\times\cdots\times E(\psi_m).$$
For any $x\in\mathcal{E}_{\infty}$, by the construction of annuli in each coordinate, it is clear that
$x_i\in W(\psi_i)$ for each $i=1,\ldots, m$.
For $i=1,\ldots,m$, it suffices to show that for any $c<1$, $x_i\notin W_i(c\psi_i)$.

Let $t_0\ge 1$ be a large enough integer such that $c_{t_0}>c$. For any $q>Q_{(t_0-2)m+i}$, there exists $t\ge t_0$ such that
\begin{align*}
Q_{(t-2)m+i}\le q<Q_{(t-1)m+i}.
\end{align*}
The balls constructed at the $[(t-1)m+i]$-th level have been constructed disjoint from the $c_t\psi_i(q)$-neighborhoods of rationals whose denominators is $q$. Since $c_{t_0}<c_t$, these balls also avoid those rationals at the smaller $c_{t_0}\psi_i(q)$-neighborhood. Consequently, each component $x_i\notin W(c_{t_0}\psi_i)$.  It then follows, $x_i\notin W(c\psi_i)$ for each $i=1,\ldots,m$ and any $c<1$.
\section{Mass distribution}
In this section, we define a probability measure $\mu$ supported on $\mathcal{E}_{\infty}$, where at each level the measure is uniformly distributed among the balls (or rectangles) according to their cardinality.

We begin to define the measure of balls in $\mathcal{E}_1$. Let $B_0$ be defined as in Section \ref{sec3} and $\mu(B_0)=1$.
\begin{enumerate}
    \item For each rectangle $R_{1}\in\mathcal{R}(B_0)$, we define
    \begin{align*}
    \mu({R}_{1})=\frac{1}{\#\mathcal{R}(B_0)}.
    \end{align*}
    \item For any $R_1\in\mathcal{R}(B_0)$, for each ball $B_1\in\mathcal{C}(R_1)$, we define
    \begin{align*}
        \mu({B}_{1})=\frac{1}{\#{\mathcal{C}}({R}_1)}\cdot\mu(R_1).
    \end{align*}
\end{enumerate}

Assume that the measure of balls in $\mathcal{E}_{n-1}$ for $n\ge 2$ has been defined, we now define the measure of balls in $\mathcal{E}_{n}$. Write $n=tm+i,$ where $1\le i\le m$ and $t\in\mathbb{N}$. For each $B_{n-1}$ belongs to $\mathcal{E}_{n-1}$,

\begin{enumerate}
    \item For each rectangle $R_{n}\in\mathcal{R}(B_{n-1})$, we define
    \begin{align*}
    \mu({R}_{n})=\frac{1}{\#\mathcal{R}(B_{n-1})}\cdot\mu(B_{n-1}).
    \end{align*}
    \item For any $R_n\in\mathcal{R}(B_{n-1})$, for each ball $B_n\in\mathcal{C}(R_n)$, we define
    \begin{align*}
        \mu({B}_{n})=\frac{1}{\#{\mathcal{C}}({R}_n)}\cdot\mu(R_n).
    \end{align*}
\end{enumerate}
By Kolmogorov's consistency theorem, the set function $\mu$ can be uniquely extended into a probability measure supported on $\mathcal{E}_{\infty}$. 

Additionally, for any $B_n\in\mathcal{E}_n$, the construction in Section \ref{sec3} provides a nested sequence
$$B_0\supset R_1\supset B_1\supset R_2\supset B_2\supset\cdots,$$
with $R_n\in\mathcal{R}(B_{n-1})$ and $B_n\in\mathcal{C}(R_n)$ for any $n\ge 1$. Let $q_n$ be the denominator of the rational point corresponding to $B_n\subset R_n$, denote $t_n=\lfloor n/m\rfloor$ and $i_n=n-mt_n$, write $\psi_{i_n}=\psi_i$ for $i\equiv i_n \ (\text{mod } m)$ in the representation system $i\in \{1,2,\ldots,n\}$. Combined with (\ref{est1}) and (\ref{est2}), we have that
\begin{align}\label{ieq:iter}
\mu(B_n)&=(1-c_{t_n+1})^{m-1}\psi_{i_n}(q_n)^{m-1}Q_n^{-2}\cdot\frac{\mu(B_{n-1})}{\mathcal{L}^m(B_{n-1})}.
\end{align}
Therefore,  we conclude that
$$\frac{\log\mu(B_{n})}{\log \gamma(B_{n})}=\underbrace{\frac{\log ((1-c_{t_n+1})^{m-1}\psi_{i_n}(q_n)^{m-1}Q_n^{-2})}{\log((1-c_{t_n+1})\psi_{i_n}(q_n))}}_{\text{denote by (I)}}+\underbrace{\frac{\log (\mu(B_{n-1})/\mathcal{L}^m(B_{n-1}))}{\log((1-c_{t_n+1})\psi_{i_n}(q_n))}}_{\text{denote by (II)}}.$$
For the term $(\text{I})$, by the monotonicity of $\psi_i$ and the choice of $Q_n$ in (\ref{choiceQn}), we have following estimate
   $$ (\text{I})\ge m-1+\frac{\log Q_n^{-2}}{\log( (1-c_{t+1})\psi_i(Q_n))}\ge m-1+\min_{1\le i\le m}\left\{\frac{2}{\alpha_i}\right\}.$$
For the term $(\text{II})$, by iterating (\ref{ieq:iter}), one concludes that
\begin{align*}
    (\text{II})&\ge \frac{\sum_{j=1}^{n-1}\log ((1-c_{t_j+1})^{-1}\psi_{i_j}(q_j)^{-1}Q_j^{-2})+\log {\gamma(B_0)^{-m}}}{\log((1-c_{t_n+1})\psi_{i_n}(q_n))}-o(1).
\end{align*}
By (\ref{condithm1}), the monotonicity of $\psi_i$ and the choice of $Q_n$ in (\ref{choiceQn}), it follows that $(\text{II})$ tends to $0$. Therefore, one has
\begin{align}\label{result_special}
    \liminf_{n\rightarrow\infty}\frac{\log\mu(B_{n})}{\log \gamma(B_{n})}\ge m-1+\min_{1\le i\le m}\left\{\frac{2}{\alpha_i}\right\}:=s_0.
\end{align}

\section{Proof of Theorem \ref{thm:main}}\label{sec:proof}
Consider a general ball $$B(x,r)=\prod_{i=1}^{m}B(x_i,r),$$ where $x\in\mathcal{E}_{\infty}$ and $r$ sufficiently small. Together with Lemma \ref{MDP}, we finish the proof of Theorem \ref{thm:main} by estimating the $\mu$ measure of $B(x,r)$.

Let $n$ be the smallest integer for which the ball $B(x,r)$ can intersect only one ball in $\mathcal{E}_{n-1}$ but intersects at least two balls in $\mathcal{E}_{n}$, we write $n=tm+i$ for some $1\le i\le m$ and $t\in\mathbb{N}$. The measure of $B(x,r)$ will be discussed below in several parts.
\begin{itemize}
\item[(1)] If $r\ge \gamma(B_{n-1})$, then we have that
$$\mu(B(x,r))\le\mu(B_{n-1}).$$
Therefore, for any $\epsilon>0$ and $r$ small enough, by (\ref{result_special}), it follows that 
\begin{align}\label{r1}
\frac{\log\mu(B(x,r))}{\log r}\ge\frac{\log \mu(B_{n-1})}{\log \gamma(B_{n-1})}\ge s_0-\epsilon.
\end{align}
\item[(2)] If $Q_n^{-2}\le r<\gamma(B_{n-1})$, the ball $B(x,r)$ can intersect at least two lifted rectangles contained in $B_{n-1}$. Let $R_n$ be a such lifted rectangle which intersect the ball $B(x,r)$, we have that
    $$R_n\cap B(x,r)\subset \prod_{j=1}^{i-1}B(x_j,2r)\times R_n\vert_i\times \prod_{j=i+1}^{m}B(x_j,2r).$$
    Then, combined with the volume argument, the number of balls $B_n$ in $R_n$ which intersect $B(x,r)$ is at most
    $$\prod_{\substack{1\le j\le m\\j\neq i}}\frac{2r}{|R_n\vert_i|}\asymp \left(\frac{r}{(1-c_{t+1})\psi_i(q)}\right)^{m-1},$$
    where $q$ is the denominator of the rational point associated with $R_n$. Consequently, the measure of $B(x,r)\cap R_n$ is bounded by
    \begin{equation}\label{single}
    \begin{split}
    \mu(B(x,r)\cap R_n)&\le\sum_{\substack{B_n\subset R_n:B_n\cap B(x,r)\neq\emptyset}}\mu(B_n)\\
    &\le \#\big\{B_n\subset R_n:B_n\cap B(x,r)\neq\emptyset\big\}\cdot\mu(B_{n})\\
    &\le r^{m-1}\cdot Q_n^{-2}\cdot\frac{\mu(B_{n-1})}{\mathcal{L}^m(B_{n-1})},
    \end{split}
    \end{equation}
    where $\mu(B_n)$ is equal for any $B_n\subset R_n$ in second inequality. In this case, the number of lifted rectangles in $\mathcal{R}(B_{n-1})$ that intersect with $B(x,r)$ is asymptotically $r\cdot Q_n^2$. Therefore, we have
    \begin{align*}
    \mu(B(x,r))&\le\sum_{R_n\subset \mathcal{R}(B_{n-1}):R_n\cap B(x,r)\neq\emptyset}\mu(B(x,r)\cap R_n)\\
    &\lesssim r\cdot Q_n^2\cdot r^{m-1}\cdot Q_n^{-2}\cdot\frac{\mu(B_{n-1})}{\mathcal{L}^m(B_{n-1})}\\
    &\le r^{m}\cdot\frac{\mu(B_{n-1})}{\mathcal{L}^m(B_{n-1})}.
    \end{align*}
    It implies that for any $\epsilon>0$ and $r$ small enough,
    \begin{equation}\label{r2}
    \begin{split}
    \frac{\log\mu(B(x,r))}{\log r}&\ge m+\frac{\log (\mu(B_{n-1})/\mathcal{L}^m(B_{n-1}))}{\log r}\\
    &\ge m+\frac{\log( \mu(B_{n-1})/\mathcal{L}^m(B_{n-1}))}{\log Q_n^{-2}}\\
    &\ge m-\epsilon,
    \end{split}
    \end{equation}
    where the last inequality holds by the choice of $Q_n$.
    \item[(3)] If $r<Q_n^{-2}$, the ball can only intersect one lifted rectangle $R_n$ in $B_{n-1}$. Let $p/q$ be the rational point associated with that rectangle. Since the ball $B(x,r)$ 
     intersects two balls in $\mathcal{E}_n$, we have $r>(1-c_{t+1})\psi_i(q)$. 
    Recall that in (\ref{single}), we have that
        $$\mu(B(x,r))=\mu(B(x,r)\cap R_n)\le r^{m-1}\cdot Q_n^{-2}\cdot\frac{\mu(B_{n-1})}{\mathcal{L}^m(B_{n-1})}.$$
        Therefore, for any $\epsilon>0$ and $r$ small enough, we have that
        \begin{equation}\label{r3}
        \begin{split}
        \frac{\log\mu(B(x,r))}{\log r}&\ge m-1+\frac{-2\log Q_n+\log( \mu(B_{n-1})/\mathcal{L}^m(B_{n-1}))}{\log ((1-c_{t+1})\psi_i(Q_n))}\\
        &\ge m-1+\frac{2}{\alpha_i}-\epsilon,
        \end{split}
        \end{equation}
        by the choice of $Q_n$ and the definition of $\alpha_i$.
\end{itemize}

In other words, we conclude that
$$\liminf_{r\rightarrow 0}\frac{\log \mu(B(x,r))}{\log r}\ge s_0.$$
By Lemma \ref{MDP}, one immediately obtain that
$$\dimH E(\psi_1)\times\cdots\times E(\psi_m)\ge s_0.$$
The upper bound
\[
\dimH\bigl(E(\psi_1)\times\cdots\times E(\psi_m)\bigr)\leq s_0
\]
follows immediately from the inclusion
\[
E(\psi_1)\times\cdots\times E(\psi_m)
\subset
W(\psi_1)\times\cdots\times W(\psi_m),
\]
together with the result of Wang-Wu as given in \eqref{eq:e2}.

\medskip

\noindent{\bf Acknowledgments.} The first and third-named authors are supported by the National Key R\&D Program of China (No. 2024YFA1013700) and Natural Science Foundation of China (No. 12331005). We thank Dr Benjamin Ward for useful discussions. Part of this work was done when Mumtaz and Johannes visited the MATRIX Research Institute (Victoria). We thank the institute for its hospitality and funding.



\begin{thebibliography}{10}

\bibitem{BakerWard2025}
S.~Baker and B.~Ward.
\newblock A quantitative framework for sets of exact approximation order by
  rational numbers.
\newblock 2025.
\newblock Preprint, arXiv:2510.18451.

\bibitem{BandiDeSaxce2025}
P.~Bandi and N.~de~Saxc\'e.
\newblock Hausdorff dimension and exact approximation order in {$\Bbb R^n$}.
\newblock {\em Ann. Sci. \'Ec. Norm. Sup\'er. (4)}, 58(4):831--853, 2025.

\bibitem{BandiGhoshNandi}
P.~Bandi, A.~Ghosh, and D.~Nandi.
\newblock Exact approximation order and well-distributed sets.
\newblock {\em Adv. Math.}, 414:Paper No. 108871, 19, 2023.

\bibitem{BeresnevichDickinsonVelani}
V.~Beresnevich, D.~Dickinson, and S.~Velani.
\newblock Sets of exact `logarithmic' order in the theory of {D}iophantine
  approximation.
\newblock {\em Math. Ann.}, 321(2):253--273, 2001.

\bibitem{BDV}
V.~Beresnevich, D.~Dickinson, and S.~Velani.
\newblock Measure theoretic laws for lim sup sets.
\newblock {\em Mem. Amer. Math. Soc.}, 179:no. 846, x+91 pp., 2006.

\bibitem{BroderickKleinbock}
R.~Broderick and D.~Kleinbock.
\newblock Dimension estimates for sets of uniformly badly approximable systems
  of linear forms.
\newblock {\em Int. J. Number Theory}, 11(7):2037--2054, 2015.

\bibitem{Bugeaud2003}
Y.~Bugeaud.
\newblock Sets of exact approximation order by rational numbers.
\newblock {\em Mathematische Annalen}, 327(1):171--190, 2003.

\bibitem{Bugeaud2008}
Y.~Bugeaud.
\newblock Sets of exact approximation order by rational numbers. {II}.
\newblock {\em Uniform Distribution Theory}, 3(2):9--20, 2008.

\bibitem{BugeaudMoreira2011}
Y.~Bugeaud and C.~G. Moreira.
\newblock Sets of exact approximation order by rational numbers {III}.
\newblock {\em Acta Arithmetica}, 146(2):177--193, 2011.

\bibitem{DRV2}
M.~M. Dodson, B.~P. Rynne, and J.~A.~G. Vickers.
\newblock Metric {D}iophantine approximation and {H}ausdorff dimension on
  manifolds.
\newblock {\em Math. Proc. Cambridge Philos. Soc.}, 105(3):547--558, 1989.

\bibitem{Erdos1962}
P.~Erd\"os.
\newblock Representation of real numbers as sums and products of {L}iouville
  numbers.
\newblock {\em Michigan Mathematical Journal}, 9:59--60, 1962.

\bibitem{Falconer_book2013}
K.~Falconer.
\newblock {\em Fractal geometry}.
\newblock John Wiley \& Sons, Ltd., Chichester, 2014.
\newblock Mathematical foundations and applications.

\bibitem{KLWZ}
H.~Koivusalo, J.~Levesley, B.~Ward, and X.~Zhang.
\newblock The dimension of the set of {$\psi$}-badly approximable points in all
  ambient dimensions: on a question of {B}eresnevich and {V}elani.
\newblock {\em Int. Math. Res. Not. IMRN}, (14):10822--10843, 2024.

\bibitem{Marstrand}
J.~M. Marstrand.
\newblock Some fundamental geometrical properties of plane sets of fractional
  dimensions.
\newblock {\em Proc. London Math. Soc. (3)}, 4:257--302, 1954.

\bibitem{MauldinUrbanski4}
R.~D. Mauldin and M.~Urba{\'n}ski.
\newblock Conformal iterated function systems with applications to the geometry
  of continued fractions.
\newblock {\em Trans. Amer. Math. Soc.}, 351:no. 12, 4995--5025, 1999.

\bibitem{Joh2023RM}
J.~Schleischitz.
\newblock Metric results on sumsets and {C}artesian products of classes of
  {D}iophantine sets.
\newblock {\em Results Math.}, 78(6):Paper No. 215, 34, 2023.

\bibitem{Joh2025MZ}
J.~Schleischitz.
\newblock The set of {$\Phi $} badly approximable matrices has full {H}ausdorff
  dimension.
\newblock {\em Math. Z.}, 309(4):Paper No. 73, 18, 2025.

\bibitem{Tricot1982}
C.~Tricot, Jr.
\newblock Two definitions of fractional dimension.
\newblock {\em Math. Proc. Cambridge Philos. Soc.}, 91(1):57--74, 1982.

\bibitem{WangWu19}
B.~Wang and J.~Wu.
\newblock Mass transference principle from rectangles to rectangles in
  {D}iophantine approximation.
\newblock {\em Math. Ann.}, 381(1-2):243--317, 2021.

\bibitem{WangWuTAMS}
B.~Wang and J.~Wu.
\newblock Hausdorff dimension of the {C}artesian product of limsup sets in
  {D}iophantine approximation.
\newblock {\em Trans. Amer. Math. Soc.}, 377(5):3727--3748, 2024.

\end{thebibliography}

\end{document}